\documentclass[11pt,a4paper]{amsart}
\usepackage{mathrsfs}
\usepackage{amsfonts}
\usepackage{color}
\usepackage{pgfplots}
\usepackage[colorlinks=true,linkcolor=blue, urlcolor=red, citecolor=blue]{hyperref}	%
\makeatletter \@addtoreset{equation}{section} \makeatother

\renewcommand\thetable{\thesection.\@arabic\c@table}

\theoremstyle{plain}
\newtheorem{theorem}{Theorem}[section]
\newtheorem{lemma}[theorem]{Lemma}
\newtheorem{proposition}[theorem]{Proposition}
\newtheorem{maintheorem}{Theorem}

\newtheorem{mainproblem}{Problem}
\theoremstyle{definition}
\newtheorem{definition}[theorem]{Definition}
\newtheorem{example}[theorem]{Example}

\newtheorem{remark}[theorem]{Remark}

\numberwithin{equation}{section}

\newcommand{\R}{\mathbb{R}}
\newcommand{\N}{\mathbb{N}}
\newcommand{\eps}{\varepsilon}
\newcommand{\rn}{r^{\frac{1}{d}}}
\newcommand{\sn}{s^{\frac{1}{d}}}
\newcommand{\wumi}{W^{1;\infty,p}(\Omega,\mathbb{R}^d)}

\newcommand{\lip}[1]{\textrm{Lip}_\lambda(#1)}
\newcommand{\lipp}[1]{\overline{\lip{#1}}^{\,p}}

\newcommand{\ml}[1]{\mathcal{M}_\lambda(#1)}
\newcommand{\mlp}[1]{\mathcal{M}^{1,p}_\lambda(#1)}
\newcommand{\dfx}[2]{\frac{\partial #1}{\partial{#2}}}
\DeclareMathOperator*{\supess}{sup\,ess}

\title[Genericity of ergodicity for Sobolev ]{Genericity of ergodicity for Sobolev homeomorphisms}

\author[A. Azevedo]{Assis Azevedo}
\address{Centro de Matem\'{a}tica, Universidade do Minho,
	Campus de Gualtar,
	4710-057 Braga, Portugal}
\email{assis@math.uminho.pt}

\author[D. Azevedo]{Davide Azevedo}
\address{Centro de Matem\'{a}tica, Universidade do Minho,
	Campus de Gualtar,
	4710-057 Braga, Portugal}
\email{davidemsa@gmail.com}

\author[M. Bessa]{M\'{a}rio Bessa}

\address{Departamento de Ciencias e Tecnologia and CMUP, Universidade Aberta,
Rua do Amial 752, 4200-055 Porto, Portugal}
\email{mario.costa@uab.pt}

\author[M. J. Torres]{Maria Joana Torres}
\address{CMAT  e Departamento de Matem\'{a}tica, Universidade do Minho,
	Campus de Gualtar,
	4710-057 Braga, Portugal}
\email{jtorres@math.uminho.pt}

\begin{document}


\begin{abstract}
	In this paper we obtain a weak version of Lusin's theorem in the Sobolev-$(1,p)$  uniform closure of volume preserving Lipschitz homeomorphisms on closed and connected $d$-dimensional manifolds, $d \geq 2$ and $0<p<1$. With this result at hand we will be able to prove that the ergodic elements are generic. This establishes a version of Oxtoby and Ulam theorem for this Sobolev class. We also prove that, for $1\leq p<d-1$, the topological transitive maps are generic.
\end{abstract}

\maketitle
\section{Introduction}

\subsection{To be or not to be ... ergodic}
Back in the nineteen century and during his studies in the dynamical theory of gases, Ludwig Boltzmann formulated
a principle which is fundamental in statistical physics - \emph{the ergodic hypothesis}. This remarkable work settled the foundations of ergodic theory, an area with profound connections with analysis, probability, dynamical systems and physics. 
Roughly speaking Boltzmann's principle says that \emph{time averages} equal \emph{space
	averages} at least for typical points. \emph{Ergodicity} can be formalized by saying that an automorphism $f\colon X\rightarrow X$ on a measure space $(X,\lambda)$, which leaves a measure $\lambda$ invariant, must satisfy the following
equality
\begin{equation*}
	\underset{n\rightarrow+\infty}{\lim}\frac{1}{n}\sum_{i=0}^{n-1}\varphi(f^i(x))=\int_{X}\varphi \, d\lambda,
\end{equation*}
for $\lambda$-a.e. $x\in X$ and any continuous observable $\varphi\colon X\rightarrow
\mathbb{R}$.

Another equivalent definition of ergodicity says that any
$f$-invariant set must have zero or full $\lambda$-measure. Along this paper we will be interested in the case where $\lambda$ is the volume measure. Whenever we consider perturbations of volume preserving maps (also called \emph{conservative}), we assume that they still keep the volume invariant.  We say that a property is \emph{generic} (or \emph{residual}) if it holds in a set which contains a countable intersection of open and dense subsets. In a Baire space a generic set is dense.

Some central questions in conservative dynamics include determining whether a given map is ergodic, stably ergodic, or becomes ergodic after a small perturbation. A map is stably ergodic if ergodicity prevails even after small perturbations. The answers to these questions can vary significantly depending on the topology used to compare maps. In fact:

\begin{itemize}
	\item [(i)]  The ergodic conservative homeomorphisms are $C^0$-generic. This was proved in the celebrated paper by Oxtoby and Ulam \cite{OU}. In particular, any $C^0$ homeomorphism can be $C^0$ approximated by an ergodic one. Yet, there is no possible way to get $C^0$-open sets of ergodic homeomorphisms, since the $C^0$-topology as being really very feeble allows us to destroy ergodic maps under $C^0$-perturbations. The techniques to get $C^0$-genericity of ergodicity in \cite{OU} were pushforward by Krengel in \cite{K} to get $C^0$-genericity of zero metric entropy;
	\item [(ii)] there are examples of $C^1$-open sets of conservative diffeomorphisms, the Anosov ones, such that all the elements of the open set are ergodic with respect to volume \cite{A}. In this case, we have stable ergodicity and stable positive metric entropy. However, Anosov maps are far from being typical. Furthermore, trying to adapt the arguments in \cite{OU} to the $C^1$ topology and obtain $C^1$-genericity of ergodicity is doomed to failure because, similarly, we could also obtain $C^1$-genericity of zero metric entropy, which is false as we saw above;
	\item [(iii)] there are examples of $C^r$-open sets of diffeomorphism ($r>3$) such that there are no ergodic maps in the open set. These examples, eventually with more regularity, are known since the mid-20th century from the outstanding work by Kolmogorov, Arnold and Moser. This tuning on the regularity for $C^r$ ($r>3$) was provided by Yoccoz~\cite{Y}. These KAM examples ended Boltzmann's dream on fulfilling the ergodic hypothesis for most maps.
\end{itemize}

In the title of the Oxtoby and Ulam paper \cite{OU} ergodicity is referred as the outdated expression \emph{metrical transitivity}. This can be contextualized noticing that ergodicity displays a kind of measure indecomposability of the phase space. There is another type of indecomposability with topological flavour called \emph{topological transitivity} which means that the dynamical system displays a dense orbit in the whole manifold. Regarding the generic theory of topological transitivity we have that both:
\begin{itemize}
	\item [(iv)] $C^0$-generic conservative homeomorphisms are topological transitive \cite{AP} and
	\item [(v)] $C^1$-generic  conservative diffeomorphisms are topological transitive \cite{BC}. 
\end{itemize}
Throughout the article, $X$ is a smooth closed connected Riemannian manifold 
of dimension $d\geq 2$.

\bigskip

\subsection{The Sobolev class}
Items (i)--(iii) above reveal that there is a clear change of paradigm from $C^0$ to $C^1$ and from $C^1$ to $C^3$. We are interested in complementing the literature concerning the genericity of ergodicity, focusing on the first jump. Moreover, items (iv) and (v) above show that topological transitivity prevails generically both $C^0$ and $C^1$.   We will also extend these generic results to the following Sobolev framework.

Firstly we need to define, given an open bounded subset $\Omega$ of $\R^d$, a Baire space with a topology capturing the uniform and the $(1,p)$-Sobolev topologies ($0<p<\infty$), where we will look for the ergodic or the topologically transitive elements.  With this in mind, we consider $\wumi$, the completion of the space of the $C^1$ functions, continuous on $\overline{\Omega}$, admitting derivatives in $L^p(\Omega)$, relatively to the norm $\|\ \cdot\ \|_{1;\infty,p}=\max\left\{\|\ \cdot\ \|_\infty,\|\ \cdot\ \|_{1,p}\right\}$. Formally, an element  $f\in\wumi$ is represented by a pair $(\pi_1(f),\pi_2(f))$, with $\pi_1(f)$ being the uniform limit of a sequence and $\pi_2(f)$ the limit in $L^p(\Omega)^{d^2}$  of the sequence of their derivatives.\vskip1mm

If $p\geq 1$, then a pair in $\wumi$ is identified by its first coordinate as its derivative is the second one. With this in mind, it can be shown that $\wumi$ is formed by the continuous functions from $\overline{\Omega}$ to $\R^d$ admitting distributional derivatives in $L^p(\Omega)$. Contrarily to the $C^0$ and $C^1$ case, if $f,g\in \wumi$ and $f(\overline{\Omega})\subseteq\overline{\Omega}$, then $g\circ f$  may not belong to $\wumi$. For that reason, we consider $\lipp{\Omega}$ the closure of the conservative Lipschitz homeomorphisms. Finally, we consider $\mlp{\Omega}$ the space of the homeomorphisms in $\lipp{\Omega}$, which is a Baire space.\vskip1mm

If $0<p<1$, the situation is different as if $f\in \wumi$, $\pi_2(f)$ may not be the derivative of $\pi_1(f)$. For this reason, we have to adapt the definitions above, essentially calling a pair a bijection/conservative/Lipschitz/homeomorphism/ergodic if its first coordinate is.\vskip1mm

For full details see \S\ref{Sosobolev}, in particular \eqref{bijeccao}.

These concepts can be carried over to a smooth closed connected manifold.\vskip2mm

Our main goal is to obtain a version of Oxtoby-Ulam theorem  establishing that ergodicity is generic, for $0<p<1$ (Theorem~\ref{OU}), and that topological transitivity is also  generic, for $1\leq p<d-1$ (Theorem~\ref{Trans}). \\

We observe that Sobolev homeomorphisms have gained particular significance in various applications, such as certain types of PDEs in nonlinear elasticity (see the Ball-Evans Problem in \cite{IKO}), in ergodic theory (genericity of infinite topological entropy in \cite{FHT,FHT0}), in dynamical systems related to the closing lemma \cite{FHT,AABT2}, and in global analysis \cite{FH,AABT}.

\bigskip

\subsection{A volume preserving Sobolev Lusin theorem}

The Lusin theorem was summed up by J. E. Littlewood by saying that \emph{`every measurable function is nearly continuous'}. 

In  Theorem~\ref{crucial} we obtain a weak version of Lusin's theorem in the Sobolev-$(1,p)$  uniform closure of volume preserving Lipschitz homeomorphisms on closed and connected $d$-dimensional manifolds, $d \geq 2$ and $0<p<1$. 

This result generalizes for the Sobolev setting previous $C^0$ versions proved by Oxtoby~\cite{O1}, White~\cite{W}, Alpern and Prasad~\cite{AP} and Alpern and Edwards~\cite{AE}.

We point out that Theorem~\ref{crucial} is the crucial ingredient that allow us, in Theorem~\ref{OU}, to obtain a version of Oxtoby and Ulam theorem. 

\bigskip

Let $\mathcal{G}_\lambda(X)$ denote the set of measurable bijections on $X$ which keep $\lambda$ invariant. 

\begin{maintheorem}[Volume preserving Sobolev weak Lusin theorem]\label{crucial} 
	Let $X$ be a closed connected $d$-dimensional manifold, $d \geq 2$ and $0<p<1$.
	Let  $g \in \mathcal{G}_\lambda(X)$, $h\in\lipp{X}$ and consider $\delta>0$.
	Then given any weak topology neighbourhood $\mathcal{W}$ of $g$, there exists
	$f \in\lip{X}$ such that  $f \in \mathcal{W}$ and 
	\begin{equation*}
		\|f-\pi_1(h)\|_\infty < \|g - \pi_1(h)\|_{\infty}+\delta,\quad \|Df-\pi_2(h)\|_p < \delta.
	\end{equation*}
\end{maintheorem}

\medskip

\subsection{Towards to Oxtoby and Ulam theorem}

Besides its intrinsic importance in analysis, the weak Lusin theorem in the volume preserving class was crucial in a proof 
given in~\cite{AP}
of the result of Oxtoby and
Ulam, that ergodicity is generic for measure preserving homeomorphisms
of compact manifolds. Therefore, 
a Sobolev version of this theorem (cf Theorem~\ref{crucial}) will be useful in the study of dynamical properties in the Sobolev class.

Our main result is the generalization of the important classic Oxtoby and Ulam's theorem for $\mlp{X}$.

We say that an element $f\in\mlp{X}$  is \emph{ergodic} if $\pi_1(f)$ is ergodic.

\begin{maintheorem}[Sobolev Oxtoby-Ulam theorem]\label{OU}
	Let $X$ be a closed connected $d$-dimensional manifold, $d \geq 2$ and $0< p<1$. 
	The set of the ergodic elements in  $\mlp{X}$ is generic.
\end{maintheorem}

We can not fail to mention to the reader the obligation to  look through the beautiful written historical background in \cite[\S1]{AP2} describing the trajectory of, undoubtedly, one of the most important results in dynamical systems.

\bigskip

\subsection{Genericity of topological transitivity} 
We say that a homeomorphism $f$ is \emph{topologically transitive} if given any open sets $U,V\subset X$ we have
\begin{equation*}
	f^n(U)\cap V\not=\emptyset\quad\text{for some $n\in\N$}.
\end{equation*}

Since ergodicity is stronger than topological transitivity, a direct corollary of Theorem~\ref{OU} states that the topologically transitive elements in $\lipp{X}$, $0< p<1$, form a residual set. However, the problem of genericity of ergodicity for $p\geq1$ remains unsolved. Our last result is a topological counterpart of the Oxtoby-Ulam theorem (topological transitivity instead of metric transitivity) for $\mlp{X}$, $p\geq 1$.

\begin{maintheorem}[Sobolev generic transitivity]\label{Trans}
	Let $X$ be a closed connected $d$- dimensional manifold, $d \geq 3$ and $1\leq p<d-1$. 
	The topologically transitive elements in $\mlp{X}$ form a generic set.
\end{maintheorem}

\bigskip

\subsection{Open problems}
We end the introduction by considering several important open questions. In \cite{FHT0}, the early 1980s result by Yano \cite{Ya} was generalized for Sobolev homeomorphisms. We ask if an analogous result holds in the conservative class posing the following:
\begin{mainproblem}
	Is infinite topological entropy generic in $\mlp{X}$ for $p>0$?
\end{mainproblem}

The seminal ideas from Oxtoby and Ulam paper \cite{OU} on periodic approximations can be used to obtain $C^0$-genericity of zero metric entropy (see \cite{K,R,CP}). With this in mind we consider the following:
\begin{mainproblem}
	Is zero metric entropy generic in $\mlp{X}$ for $p>0$?
\end{mainproblem}
An affirmative answer to previous problems would allow us to conclude that the volume measure could not play an effective role on fulfilling an (hypothetical) variational principle for a generic set of Sobolev conservative homeomorphisms.\\

Finally, we pose what we believe to be the most important open question:
\begin{mainproblem}
	Is ergodicity generic in $\mlp{X}$ for $p\geq1$?
\end{mainproblem}

This paper is organized as follows. In Section~\ref{pre} we introduce the space of automorphisms and the Sobolev-$(1,p)$ framework in which we will work ($0<p<\infty$). In Section~\ref{PST} we prove a key perturbation result (Theorem \ref{key:pert2}) which will be crucial to prove the theorems of the following sections.
In Section~\ref{PLusin} we obtain a weak version of Lusin's theorem for $p<1$ (Theorem~\ref{crucial}). In Section~\ref{PMR}, we prove the Sobolev Oxtoby-Ulam theorem (Theorem~\ref{OU}) with respect to the same topology. In  Section~\ref{PTT} we prove the genericity of topological transitivity (Theorem~\ref{Trans}) for $1\leq p<d-1$.
Finally, in Section \ref{example} we construct an explicit example showing that the space $\mlp{X}$ is not a  functional one.

\section{Preliminaries}\label{pre}

We recall that $X$ is a smooth closed connected Riemannian manifold 
of dimension $d$ and $d_X$ is the distance on $X$ induced by the Riemannian structure.
We denote the Euclidean norm in $\R^d$ by $|\cdot|$.
Let $\lambda$ stand for the measure on both $X$ and $\R^d$ derived from the standard volume form.

\subsection{Automorphisms and homeomorphisms of $(X,\lambda)$} 

An {\em automorphism} of the underlying Borel measure space 
$(X,\lambda)$ is a bijection $g \colon X \to X$ such that $g$ and, consequently,  $g^{-1}$ are measurable functions and 
$\lambda(B)=\lambda(g^{-1}(B))$
for all measurable sets $B$.
Automorphisms which differ on a set of measure zero will be identified.
We  denote by $\mathcal{G}_\lambda(X)$ the space of automorphisms of $(X,\lambda)$. 
We shall consider two topologies on $\mathcal{G}_\lambda(X)$. The first one is the 
{\em weak topology} which is determined by the metric 
\begin{equation*}
	\rho(f,g)=\inf\{\delta\colon \lambda\{x\colon d_X(f(x),g(x))\geq \delta\}<\delta\}.
\end{equation*}

Define also the {\em uniform topology} by the metric 
\begin{equation*}
	\|f-g\|_{\infty}= \supess_{x \in X} d_X(f(x),g(x)).
\end{equation*}

The space $\mathcal{G}_\lambda(X)$  is 
topologically complete with the weak topology (see~\cite{H,AP}) and complete  with the uniform topology. Thus, with each of these topologies,  $\mathcal{G}_\lambda(X)$ is a Baire space. Moreover, $\mathcal{G}_\lambda(X)$ endowed with the weak topology is a topological group (cf \cite[pp. 39]{AP}).

We denote by $\ml{X}$ the subspace of all homeomorphisms in $\mathcal{G}_\lambda(X)$, 
endowed with the uniform topology. 
This space is topologically complete (see~\cite{O}), thus a Baire space.
We shall call {\em volume preserving homeomorphisms} of $X$ the elements in 
$\ml{X}$. For $f\in \ml{X}$ and $\delta>0$ we shall denote by $B^\infty(f,\delta)$ the \emph{uniform ball} centred in $f$ and with radius $\delta$.

\bigskip

\subsection{Sobolev maps}\label{Sosobolev} Let $\Omega$ be an open bounded subset of $\R^d$ and  $0<p <\infty$. It is well known that $L^p(\Omega)$ endowed with its natural norm or quasi-norm for  $1\leq p< \infty$ or $0<p<1$, respectively, defined by 
\begin{equation*}
	\| f\|_p=\left(\int_\Omega|f|^pd\lambda\right)^\frac{1}{p}
\end{equation*}
is a complete space.

If $0<p<1$, we have the inequality
\begin{equation}\label{pq}
	\forall f,g\in L^p(\Omega)\quad \|f+g\|^p_p\leq \|f\|^p_p+\|g\|^p_p,
\end{equation}
proving that  $L^p(\Omega)$ is a metric space. Contrarily to the case where $1\leq p<\infty$, if $0<p<1$, $L^p(\Omega)$ is neither a normable nor a locally convex space, from where one can obtain that its dual is the trivial space.

Of course, if $\Omega$ is bounded and $0<p<q<\infty$ then  $L^q(\Omega)$ is continuously included in $L^p(\Omega)$. In fact, if $f\in L^q(\Omega)$ we have, using the H\"older inequality,
\begin{equation}\label{pqmenor}
	\|f\|_p\leq
	\lambda(\Omega)^{\frac{q-p}{pq}}\|f\|_q.
\end{equation}

For $1\leq p<\infty$ there are (usually) two equivalent ways to define the Sobolev spaces $W^{n,p}(\Omega)$  \cite{H=W}): one using a completion of $C^n(\Omega)$ and other using weak derivatives of functions in $L^p(\Omega)$. If $0<p<1$, only the first one has a sense because, for those $p$, a function in $L^p(\Omega)$ may not be locally integrable. 

For $0<p<1$ the completion $W^{n,p}(\Omega)$ has a strange behaviour. For simplicity, let us look at the case where $n=d=1$,  $\Omega$ being an interval. Peetre \cite{peetre} proved that for all $g,h\in L^p(\Omega)$ there exists an element of $W^{1,p}(\Omega)$ represented by a sequence $(f_n)_n$ such that
\begin{equation*}
	\lim_n f_n=g, \quad\lim_nf'_n=h.
\end{equation*}

That is, in this case, the elements of $W^{1,p}(\Omega)$ are represented by pairs $(g,h)\in L^p(\Omega)\times L^p(\Omega)$, where $h$ can be any function in $L^p(\Omega)$ and not necessarily the weak derivative of $g$, as happens if $p\geq 1$. 

Just to give an idea of what is possible in this space, we present an explicit example (adapted  from \cite{peetre}) of  this situation.

\begin{example} \label{peetre2}
	For $n\in\N$, let $a,b>0$ with $a+b=\frac{1}{n}$ and consider $f_n:[0,\frac{1}{n}]\rightarrow [0,\frac{1}{n}]$ defined by	
	\begin{equation*}
		f_n(x)=
		\left\{\begin{array}{ll}
			\frac{b}{a}\,x & \text{if $x\leq a$}\\[2mm]
			\frac{a}{b}\big(x-\frac{1}{n}\big)+\frac{1}{n} & \text{if $x> a$.}
		\end{array}\right.
	\end{equation*}

	We extend $f_n$ to a function from $[0,1]$ to $[0,1]$ defining $f_n(\frac{k}{n}+x)=f_n(x) +\frac{k}{n}$, if $x\in \left]\frac{k}{n},\frac{k+1}{n}\right]$, with $k\in\{0,1\ldots,n-1\}$ (see Figure \ref{figura}). Notice that these functions are Lipschitz homeomorphisms.

	\begin{figure}[htb]
		\centering	\begin{tikzpicture}[xscale=0.8,yscale=0.8]
			\begin{axis}[
				axis lines = left,
				]
				\addplot [ultra thick,
				domain=0:1/25, 
				samples=10, 
				color=blue,
				]
				{4*x};
				\addplot [ultra thick,
				domain=1/25:1/5, 
				samples=10, 
				color=blue,
				]
				{1/4*(x-1/5)+1/5};
				\addplot [ultra thick,
				domain=0:1, 
				samples=100, 
				color=red,
				]
				{x};
				\addplot [ultra thick,
				domain=1/5:6/25, 
				samples=10, 
				color=blue,
				]
				{4*(x-1/5)+1/5};
				\addplot [ultra thick,
				domain=6/25:2/5, 
				samples=10, 
				color=blue,
				]
				{1/4*(x-2/5)+2/5};
				\addplot [ultra thick,
				domain=2/5:11/25, 
				samples=10, 
				color=blue,
				]
				{4*(x-2/5)+2/5};
				\addplot [ultra thick,
				domain=11/25:3/5, 
				samples=10, 
				color=blue,
				]
				{1/4*(x-3/5)+3/5};
				\addplot [ultra thick,
				domain=3/5:16/25, 
				samples=10, 
				color=blue,
				]
				{4*(x-3/5)+3/5};
				\addplot [ultra thick,
				domain=16/25:4/5, 
				samples=10, 
				color=blue,
				]
				{1/4*(x-4/5)+4/5};
				\addplot [ultra thick,
				domain=4/5:21/25, 
				samples=10, 
				color=blue,
				]
				{4*(x-4/5)+4/5};
				\addplot [ultra thick,
				domain=21/25:5/5, 
				samples=10, 
				color=blue,
				]
				{1/4*(x-5/5)+5/5};
				\addplot[color = black, dashed, thick] coordinates {(1/5, 0) (1/5,1/5)};
				\addplot[color = black, dashed, thick] coordinates {(2/5, 0) (2/5,2/5)};
				\addplot[color = black, dashed, thick] coordinates {(3/5, 0) (3/5,3/5)};
				\addplot[color = black, dashed, thick] coordinates {(4/5, 0) (4/5,4/5)};
				\addplot[color = black, dashed, thick] coordinates {(1, 0) (1,1)};
			\end{axis}
		\end{tikzpicture}
		\caption{Graphic of $f_5$ (with $a=\frac{1}{25}$) and of the identity.}
		\label{figura}
	\end{figure}
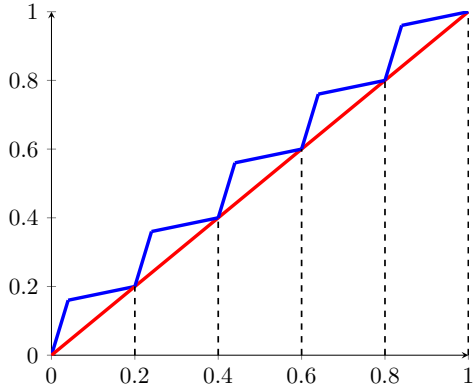

	Let $I$ be the identity function in $[0,1]$. Then $\|f_n-I\|_\infty\leq \frac{1}{n}$, proving that $ (f_n)_n$ converges uniformly and in all $L^p(I)$, for $0<p<\infty$. If we fix $0<p<1$, then 
	\begin{equation*}
		\|f'_n\|^p_p=n\int_0^{\frac{1}{n}}|f'_n(x)|^pdx=n\left(a^{1-p}b^p+b^{1-p}a^p\right)
		=(an)^{1-p}(bn)^p+(bn)^{1-p}(an)^p.
	\end{equation*}
	If we choose $a$ such that $\lim_n an=0$ (or $\lim_nbn=0$) then $(f_n)_n$ converges in $W^{1,p}(I)$ to an element $f$ represented by the pair $(I,0)$.
\end{example}

In the present paper all the functions we will be interested in are volume preserving homeomorphisms from $\overline{\Omega}$ to $\overline{\Omega}$, whose components are in $W^{1,p}(\Omega)$.

With this in mind, we start with the vectorial space
\begin{equation}\label{base}
	\left\{f=(f_1,\ldots,f_d)\in C^1(\Omega)^d\cap C^0(\overline{\Omega})^d: \tfrac{\partial f_i}{\partial x_j}\in L^p(\Omega), \text{ for } i,j=1,\ldots,d\right\}
\end{equation}
with the norm or quasi-norm, for  $1\leq p<\infty$ or $0<p<1$, respectively, defined by
\begin{equation}\label{dwf}
	\|f\|_{1;\infty,p}=\max\{\|f\|_{\infty},\|Df\|_{p}\},
\end{equation}
where $\|f\|_\infty=\max_i\|f_i\|_\infty$ and  $\|Df\|_{p}=\max_{i,j}\left\|\dfx{f_i}{x_j}\right\|_p$.

Notice that we do not need to include $\|f\|_p$ in \eqref{dwf}, as it is superfluous.

\begin{definition}
	If $\Omega$ is an open bounded subset of $\R^d$ and $0<p<\infty$, we denote by $\wumi$ the completion of the space defined in \eqref{base} relatively to the norm $\|\ \cdot\ \|_{1;\infty,p}$.
\end{definition}
As in any completion, if $f\in \wumi$ we denote by $\|f\|_{1;\infty,p}$ the limit of the sequence $(\|f_n\|_{1;\infty,p})_n$, where $(f_n)_n$ is a representative Cauchy sequence of $f$.

\begin{remark}\label{inclusion}
	Notice that, using \eqref{pqmenor}, if $0<p<q<\infty$ then 
	$W^{1;\infty,q}(\Omega,\mathbb{R}^d)$ is continuously included in $\wumi$.
\end{remark}

An element of $f\in \wumi$ is represented by a Cauchy sequence $(f_n)_n$  relatively to the norm defined in \eqref{dwf}.  Thus, the sequences $(f_n)_n$ and $(D f_n)_n$ are Cauchy (and then convergent) sequences in $C^0(\overline{\Omega})^d$ and  
$\left(L^p(\Omega)\right)^{d^2}$, respectively. Then we have a linear continuous inclusion
\begin{equation}\label{bijeccao}
	\pi=(\pi_1,\pi_2):\wumi\longrightarrow C^0(\overline{\Omega})^d\times \left(L^p(\Omega)\right)^{d^2}
\end{equation} 
As noticed before, by \cite{H=W}, if $p\geq 1$, then $\pi_2(f)$ is the weak derivative of $\pi_1(f)$ and then, $f$ will be represented only by $\pi_1(f)$. Furthermore, if $u\in L^p(\Omega)^d$ admitting weak partial derivatives in $L^p(\Omega)$, then, not only there exists a sequence of $C^\infty$ functions converging to $u$ in $W^{1,p}(\Omega,\R^d)$, but also in $C^0(\overline{\Omega})^d$, if $u\in C^0(\overline{\Omega})^d$. Then
\begin{equation}\label{wlip}
	\wumi=\left\{f=(f_1,\ldots,f_d)\in  C^0(\overline{\Omega})^d: \tfrac{\partial f_i}{\partial x_j}\in L^p(\Omega), \text{ for } i,j=1,\ldots,d\right\},
\end{equation}
were the derivatives are weak ones.

\begin{example}\label{Id0}
	If $f_n$ is defined as in Example \ref{peetre2} and 
	\begin{equation*}
		\begin{array}{rclc}
			F_n:&	[0,1]^d & \longrightarrow & [0,1]^d\\
			& 	(x_1,\ldots,x_d) & \longmapsto & \big(f_n(x_1),\ldots,f_n(x_d)\big)
		\end{array}
	\end{equation*}
	then $(F_n)_n$ is a sequence of homeomorphisms converging uniformly to the identity $I$ and  $\left(DF_n)\right)_n$ converges to the null function in $L^p(I^d)^{d^2}$. So the limit of $(F_n)_n$ in $W^{1;\infty,p}(\Omega,\R^d)$ is represented by the pair $(I,0)$ and, of course the second component is not the derivative of the first one.
\end{example} 

Using \eqref{wlip} and the Remark \ref{inclusion} it follows that the Lipschitz functions from $\overline{\Omega}$ to $\overline{\Omega}$ belong to $\wumi$  for every $0<p<\infty$. Consequently, the following definition make sense.

{
	
	\begin{definition} If $\Omega$ is a bounded open subset of $\R^d$ and $0<p<\infty$, let $\lip{\Omega}$ be the subspace of $\wumi$ formed by the volume preserving Lipschitz homeomorphisms from $\overline{\Omega}$ to $\overline{\Omega}$. The closure in $\wumi$ of $\lip{\Omega}$ will be denoted by $\lipp{\Omega}$.
	\end{definition}

	\begin{remark}\label{toto}
		If $f\in \lipp{\Omega}$ then $\pi_1(f)\in C^0_\lambda(\overline{\Omega})$, the set of volume preserving functions, but may not be an homeomorphism. But if $\pi_1(f)$ it is injective, then it is a homeomorphism.
	\end{remark}
	
	This justifies the definition of the Baire space we are going to work with.

	\begin{definition}
		If $\Omega$ is a bounded open subset of $\R^d$ and $0<p<\infty$, we denote by $\mlp{\Omega}$ the subspace of $\lipp{\Omega}$, 
		\begin{equation*}
			\left\{f\in \lipp{\Omega}: \pi_1(f)\in \ml{\overline{\Omega}}\right\}.
		\end{equation*}
	\end{definition}

	It is clear that, if we have an element $f=(g,h)\in\mlp{\Omega}$ then, viewing $h$ as a matrix $d\times d$, the  determinant of $h$ is, almost everywhere, equal to  $\pm 1$. This is a consequence of the fact that the $L^p$ convergence implies almost everywhere convergence and that $g$ is the limit in $L^p(\Omega)^{d^2}$ of a sequence in $\lip{\Omega}$. In particular the element $(I,0)$ defined in Example \ref{Id0} does not belong to $\mlp{\Omega}$.
	
	It is natural to ask if $\mlp{I^d}$ contains any element   $f=(g,h)\in\mlp{I^d}$ such that $h$ is not the weak derivative of $g$. Of course the question only makes sense if $0<p<1$. In Section \ref{example}, we answered positively to this question using techniques that will be developed in the following sections.

	\begin{proposition} If $\Omega$ is a bounded open subset of $\R^d$ and $0<p<\infty$ then $\mlp{\Omega}$ is a Baire space.
	\end{proposition}
	\begin{proof} Let 
		\begin{equation*}
			X_n=\left\{f\in\lipp{\Omega}: \forall x,y\in\overline{\Omega}\quad \big[|x-y|\geq \tfrac{1}{n}\ \Rightarrow\ \pi_1(f)(x)\neq \pi_1(f)(y)\big]\right\}.
		\end{equation*}
		Using Remark \ref{toto}, it is clear that 	$\mlp{\Omega}=\cap_n X_n$. To conclude the proof it is enough to show that, for $n\in\N$, $X_n$ is an open subset of $\lipp{X}$, which is a complete space.\vskip2mm
		
		Given $f\in X_n$ let
		\begin{equation*}
			\delta=\min\left\{\left|\pi_1(f)(x)-\pi_1(f)(y)\right|:|x-y|\geq\tfrac{1}{n}|\right\}.
		\end{equation*} 
		
		As $\overline{\Omega}$ is compact, then $\delta> 0$. Let us see that, if $g\in\lipp{\Omega}$ and $\|f-g\|_{1;\infty,p}<\frac{\delta}{2}$ then $g\in X_n$. In fact, if $x,y\in\lipp{\Omega}$ and  $|x-y|\geq\frac{1}{n}$ then, as  $\|f-g\|_\infty\leq \|f-g\|_{1;\infty,p}$,
		\begin{align*}
			\delta\leq |\pi_1(f)(x)-\pi_1(f)(y)| \leq &  |\pi_1(f)(x)-\pi_1(g)(x)|+  |\pi_1(g)(x)-\pi_1(g)(y)|+ \\
			&   |\pi_1(g)(y)-\pi_1(f)(y)|\\
			< &
			\delta+|\pi_1(g)(x)-\pi_1(g)(y)|,
		\end{align*}
		from where we conclude that $\pi_1(g)(x)\neq \pi_1(g)(y)$ and then $g\in X_n$.
	\end{proof}

	\medskip
	
	All these concepts can be carried over to a smooth closed connected manifold. 

\section{A key perturbation theorem}\label{PST}

The next theorem will serve as a cornerstone result essential for establishing a volume preserving Sobolev weak Lusin theorem, and the generic results of the following sections. Here, $I^d=[0,1]^d$ denotes the $d$-dimensional unit cube.

\begin{theorem}\label{key:pert2}
	Let $d \geq 2$, $0<p<d-1$, $N\in\N$, and consider $ \{P_i\}_{i=1}^N$ and $\{Q_i\}_{i=1}^N$ two sets of $N$ distinct interior points of $I^d$.
	
	For sufficiently small $r>0$, there is $F$, a volume preserving $C^{\infty}$ diffeomorphism of $I^d$, equal to the identity outside $I^d$ and in a neighbourhood of the boundary of $I^d$, which sends by translation a ball of radius $r$ centred in $P_i$ onto a ball centred in $Q_i$. 
	
	Moreover, if $\delta>0$, we can choose $r$ such that  	
	\begin{equation*}
		\|F-Id\|_\infty < \max_i|P_1-Q_i|+\delta,\qquad  \left\|F-Id\right\|_{1,p}<\delta.
	\end{equation*}
\end{theorem}

\begin{remark}
	In the previous theorem, the function $F$ will actually be equal to the identity outside a finite union of tubular neighbourhoods of ellipses of radius chosen as small as we need. 
\end{remark}

\begin{remark}
	The conclusion of the previous theorem is false if $p>d$, unless $P_i=Q_i$ for all $i=1,\ldots,N$. In fact, that conclusion would imply the existence of a sequence $(F_n)_n$ of functions converging to $I_d$ in $W^{1,p}(I^d,I^d)$ and such that the sequence $(\|F_n-Id\|_\infty)_n $ converges to $\max_i|P_1-Q_i|$, which contradicts the continuous inclusion of $C^0(I^d)$ in $W^{1,p}(I^d)$, for those $p$.
\end{remark}
The remainder of this section is devoted to proving Theorem~\ref{key:pert2}.

\bigskip

\subsection{Pseudo-ring}
We introduce the following notion. Let $\mathcal{E}$ be an ellipse (not reduced to a point). Then there exist unique $b$ and $R$, with $0\leq b\leq 1$ and $R> 0$, such that in a certain orthonormal basis of $\R^d$, after a translation, the equations of the ellipse are $x_1^2+\left(\tfrac{x_2}{b}\right)^2=R^2$, $x_i=0$, for $i\geq 3$.

For $r>0$, we denote by $PR^\mathcal{E}_r$ the set, in the new variables,
\begin{equation*}
	\left\{x\in \R^d: \left(\sqrt{x_1^2+\left(\frac{x_2}{b}\right)^2}-R\right)^2+\sum_{i=3}^dx_i^2\leq r^2\right\}.
\end{equation*}
We call this set a pseudo-ring in $\R^d$ as, if $b=1$ and $R>r>0$, it is  a $d$-ring i.e., obtained by the rotation of a $d-1$ ball of radius $r$ around a circle of radius $R$.  Notice that, if $R>r>0$ then, denoting by $\omega_{d-1}$  the volume of the unit ball in $\R^{d-1}$,
\begin{equation}\label{volumePT}
	\lambda(PR^\mathcal{E}_r)=
	2\pi\,R\,b\,\omega_{d-1}r^{d-1}.
\end{equation}

If $b<1$ then $PR^\mathcal{E}_r$ is not a tubular neighbourhood of $\mathcal{E}$ but is contained and contains tubular neighbourhoods of $\mathcal{E}$. We denote by $D(\mathcal{E},s)$, for $s>0$, the closed tubular neighbourhood of $\mathcal{E}$ of radius $s$.

\begin{lemma}\label{contido}
	In the previous conditions,
	\begin{equation*}
		D(\mathcal{E},br)\subseteq PR^\mathcal{E}_r\subseteq D(\mathcal{E},r).
	\end{equation*}
	
	In particular, the disk $D\left((\pm R,0,\ldots,0),br\right)$ is contained in $PR^\mathcal{E}_r$.
\end{lemma} 
\begin{proof}
	We only need to consider the case $d=2$.
	
	\noindent Let $x=(x_1,x_2)\in PR^\mathcal{E}_r$ and consider  $\xi=\sqrt{x_1^2+\left(\frac{x_2}{b}\right)^2}$, $y=\frac{R}{\xi}(x_1,x_2)\in\mathcal{E}$. Then
	\begin{equation*}
		\|x-y\|_2=\tfrac{|\xi-R|}{\xi} \sqrt{x_1^2+x_2^2}\leq |\xi-R|\leq r,
	\end{equation*}
	proving that $x\in D(\mathcal{E},r)$.

	On the other hand, let  $(x_1,x_2)\in D(\mathcal{E},br)$ and consider $(\alpha,\beta)\in\mathcal{E}$ such that $\|(x_1,x_2)-(\alpha,\beta)\|_2\leq br$. Then $(x_1,x_2)\in PR^\mathcal{E}_r$, as 
	\begin{align*}
		\left|\sqrt{x_1^2+\left(\tfrac{x_2}{b}\right)^2}-R\right| & = 
		\Big|\left\|\left(x_1,\tfrac{x_2}{b}\right)\right\|_2-\left\|\left(\alpha,\tfrac{\beta}{b}\right)\right\|_2\Big|\leq
		\left\|\left(x_1-\alpha,\tfrac{x_2}{b}-\tfrac{\beta}{b}\right)\right\|_2\\
		& =  \sqrt{(x_1-\alpha)^2+\tfrac{1}{b^2}(x_2-\beta)^2} \leq \tfrac{1}{b}\sqrt{(x_1-\alpha)^2+(x_2-\beta)^2}\\
		& =  \tfrac{1}{b}\,\|(x_1-\alpha,x_2-\beta)\|\leq r.
	\end{align*}
	As a particular case, noticing that $(\pm R,0)\in\mathcal{E}$, we obtain $D\left((\pm R,0),br\right)\subseteq PR^\mathcal{E}_r$.
\end{proof} 

\bigskip

\subsection{Proof of Theorem \ref{key:pert2}}
We start by proving some auxiliary results. 

\begin{lemma}\label{rotacao}
	Let $d\geq 2$, $b>0$,  $Z:\R^{d-1}\rightarrow\R$ be a differentiable function and consider $\alpha: \R^d \longrightarrow  \R$ defined by $\alpha(x)=Z\left(x_1^2+\left(\tfrac{x_2}{b}\right)^2,\bar x\right)$, where $\bar x=(x_3,\ldots,x_d)$.
	
	Then the function $F_b^\alpha: \R^d \rightarrow \R^d$ defined by
	\begin{equation*}
		F_b^\alpha(x)=\left(x_1\cos(\alpha(x))-\tfrac{1}{b}\,x_2\sin(\alpha(x)),b\,x_1\sin(\alpha(x))+x_2\cos(\alpha(x)),\bar x \right)
	\end{equation*}
	is a volume preserving diffeomorphism, with the same differential regularity as $Z$.
\end{lemma}
\begin{proof}
	
	Note that, writing  $F_b^\alpha=F=(F_1,\ldots,F_d)$, we have
	\begin{equation}\label{partial}
		\left\{\begin{array}{rcr}
			\dfx{F_1}{x_1}(x) & = & \cos(\alpha(x))-x_1\sin(\alpha(x))\dfx{\alpha}{x_1}(x)-\tfrac{1}{b}x_2\cos(\alpha(x))\dfx{\alpha}{x_1}(x)\\[2mm]
			\dfx{F_2}{x_2}(x) & = & \cos(\alpha(x))-x_2\sin(\alpha(x))\dfx{\alpha}{x_2}(x)+bx_1\cos(\alpha(x))\dfx{\alpha}{x_2}(x)\\[2mm]
			\dfx{F_1}{x_2}(x) & = & -\frac{1}{b}\sin(\alpha(x))-x_1\sin(\alpha(x))\dfx{\alpha}{x_2}(x)-\tfrac{1}{b}x_2\cos(\alpha(x))\dfx{\alpha}{x_2}(x)
			\\[2mm]
			\dfx{F_2}{x_1}(x) & = & b\sin(\alpha(x))-x_2\sin(\alpha(x))\dfx{\alpha}{x_1}(x)+bx_1\cos(\alpha(x))\dfx{\alpha}{x_1}(x)
		\end{array}\right.
	\end{equation}
	and then
	\begin{align*}
		\det{JF(x)} & =  \dfx{F_1}{x_1}(x)\dfx{F_2}{x_2}(x)-\dfx{F_1}{x_2}(x)\dfx{F_2}{x_1}(x)\\
		& =  1+bx_1\dfx{\alpha}{x_2}(x)-\tfrac{1}{b}x_2\dfx{\alpha}{x_1}(x)\\
		& =  1+bx_1\dfx{Z}{x_1}\left( x_1^2+\left(\tfrac{x_2}{b}\right)^2,\bar x\right) \tfrac{2}{b^2}x_2-\tfrac{1}{b}x_2\dfx{Z}{x_1}\left( x_1^2+\left(\tfrac{x_2}{b}\right)^2,\bar x\right) 2x_1\\
		& =  1.
	\end{align*}
	On the other hand, as 
	\begin{equation*}
		F_1^2(x)
		+\frac{F_2^2(x)}{b^2}=x_1^2+\left(\frac{x_2}{b}\right)^2,
	\end{equation*}
	then $\alpha(x)=\alpha(F_b^\alpha(x))$. Therefore, it is easy to see that the inverse of $	F_b^\alpha$ is $	F_b^{-\alpha}$.
\end{proof}

Given $\mu\in\,]0,1[$, let $h_\mu:\R\rightarrow\R$ be a $C^\infty$ function, equal to $\pi$ in $]-\infty,1-\mu]$, strictly decreasing in $]1-\mu,1[$ and equal to $0$ in $[1,+\infty[$. 
An explicit example of a function $h_\mu$ is defined in $]1-\mu,1[$ by 
\begin{equation*}
	h_\mu(t)=\frac{\pi \,e^{\frac{\mu}{t-1}}}{e^{\frac{\mu}{t-1}}+e^{\frac{\mu}{1-t-\mu}}}.
\end{equation*}

As $h_\mu'$ is zero outside $[1-\mu,1]$ and $h_\mu'$ is continuous then $|h_\mu'|$ is bounded by a constant  $C=C(\mu)$ As a curiosity one can verify that the best constant is $\frac{2\pi}{\mu}$.

\begin{lemma}\label{rotacao1}
	Let $d\geq 2$ and consider an ellipse $\mathcal{E}$.
	Then, for $r>0$, there exists $F:\R^d\rightarrow\R^d$, a volume preserving $C^\infty$ diffeomorphism, which is equal to the identity outside $PR^\mathcal{E}_r$, and in $PR^\mathcal{E}_\frac{r}{2}$ is a rotation of an angle $\pi$  in the plane of the ellipse whose centre is the centre of the ellipse.
	
	Moreover, for $r>0$, there exists $C$, increasing with $r$, such that $\left|\frac{\partial F_i}{\partial x_j}(x)\right|\leq \frac{C}{r}$, for all $i,j=1,\ldots,d$ and $x\in\R^d$.	 
\end{lemma}
\begin{proof}
	Using a translation followed by a linear function with determinant equal to $1$, we can suppose that the ellipse has equations $x_1^2+\left(\tfrac{x_2}{b}\right)^2=R^2$, $x_i=0$, for $i\geq 3$, for some $b>0$.
	
	Denote $h_{\frac{3}{4}}$ by $h$, consider $\alpha:\R^d\rightarrow\R$ defined by
	\begin{equation*}
		\alpha(x)=h\left(\tfrac{1}{r^2}\left[\left(\sqrt{x_1^2+\left(\tfrac{x_2}{b}\right)^2}-R\right)^2+\sum_{i=3}^dx_i^2\right]\right),
	\end{equation*}
	and $F$ defined in Lemma \ref{rotacao}, for this function $\alpha$. If $x\notin PR^\mathcal{E}_r$ then $\alpha(x)=0$ and so $F(x)=x$. If $x\in PR^\mathcal{E}_\frac{r}{2}$, then $\alpha(x)=\pi$ and $F(x)=(-x_1,-x_2,\bar x)$.
	
	For the second part, we only need to consider points in the interior of $PR^\mathcal{E}_r$. For those points, we have
	\begin{align*}
		\left|\frac{\partial \alpha}{\partial x_1}(x)\right| & \leq   \frac{\|h'\|_\infty}{r^2}\left|\sqrt{x_1^2+\left(\tfrac{x_2}{b}\right)^2}-R\right|\frac{2|x_1|}{\sqrt{x_1^2+\left(\tfrac{x_2}{b}\right)^2}}\leq\frac{2\|h'\|_\infty}{r},\\
		\left|\frac{\partial \alpha}{\partial x_2}(x)\right|	  & \leq   \frac{\|h'\|_\infty}{r^2}\left|\sqrt{x_1^2+\left(\tfrac{x_2}{b}\right)^2}-R\right|\frac{2|x_2|}{b^2\sqrt{x_1^2+\left(\tfrac{x_2}{b}\right)^2}}\leq\frac{2\|h'\|_\infty}{br},\\
		\left|\frac{\partial \alpha}{\partial x_i}(x)\right| & \leq  \frac{2\|h'\|_\infty}{r^2}|x_i|\leq\frac{2\|h'\|_\infty}{r},\qquad i \geq 3.
	\end{align*}
	
	The conditions on the partial derivatives of $F$ are simply a consequence of \eqref{partial}, the inequalities above, the inequalities $|\cos(\alpha(x))|,|\sin(\alpha(x))|\leq 1$, and the boundedness $|x_1|\leq R+r$, $|x_2|\leq b(R+r)$ and $|x_i|\leq r$, for $i\geq 3$.
\end{proof}

\bigskip

\begin{lemma}\label{rotacao2}
	Let $d\geq 2$ and $A\in\R^d$. Then, for $s>0$ and any plane $\Pi$ passing throughout $A$, there exists a volume preserving $C^{\infty}$ diffeomorphism $G:\R^d\rightarrow\R^d$, which is equal to the identity outside the sphere centred in $A$ and radius $s$, and in the sphere centred in $A$ and radius $\frac{s}{2}$ is a rotation of an angle $\pi$ relatively to the plane $\Pi$, with centre in $A$.

	Moreover, there exists $C$, independent of $s$, such that $\left|\frac{\partial G_i}{\partial x_j}(x)\right|\leq C$, for all $i,j=1,\ldots,d$ and $x\in\R^d$.	 
\end{lemma}
\begin{proof} 
	We define $\alpha$ like in the previous lemma, with $R=0$ and $b=1$ and follow the proof denoting by $G$ the function given by the lemma. As we have, in this case, $|x_1|,|x_2|\leq s$, the boundedness of the derivatives of $G$ are independent of $s$.
\end{proof}

\bigskip

Combining Lemma \ref{rotacao1} and Lemma \ref{rotacao2} we obtain the following.

\begin{proposition}\label{translacao}
	Let $d\geq 2$ and consider an ellipse $\mathcal{E}$ and $b=b(\mathcal{E})>0$.
	Then, for $r>0$, there exists a volume preserving $C^{\infty}$ diffeomorphism $H:\R^d\rightarrow\R^d$, equal to the identity outside $PR^\mathcal{E}_r$ and, if $P$ and $Q$ are the vertices of the ellipse, $H$ is a translation of the ball centred in $P$ and radius $\frac{b}{4}r$ to the ball centred in $Q$ and radius $\frac{b}{4}r$.
	
	Moreover, there exists $C$, independent of $r$, such that $\left|\frac{\partial H_i}{\partial x_j}(x)\right|\leq \frac{C}{r}$, for all $i,j=1,\ldots,d$ and $x\in\R^d$.	 
\end{proposition}

\begin{proof}
	Consider $F$ given by Lemma \ref{rotacao1} for the ellipse $\mathcal{E}$ and $G$ given by Lemma \ref{rotacao2} for $A=Q$, $\Pi$ the plane of the ellipse and, making attention to Lemma \ref{contido}, $s=\frac{b}{2}r$. Then $H=G\circ F$ satisfies the mentioned conditions.
\end{proof}

At last, we are in a position to prove Theorem~\ref{key:pert2}. We highlight that we construct explicit perturbations which is a novelty when compared with the proof of \cite[Theorem 2.4]{AP}. This is crucial since in the current setting we need to control their derivatives.

\begin{proof} We can suppose that $P_i\neq Q_i$ for all $i$.
	
	{\bf Case 1.} Suppose that there are ellipses $\mathcal{E}_i$ contained in the interior of $I^d$, with vertices $P_i$ and $Q_i$ in such a way that $\mathcal{E}_i\cap \mathcal{E}_j=\emptyset$, for $i\neq j$. It is clear that this always happens if $d\geq 3$. 
	
	As the ellipses are compact  subsets of $\R^d$, there exists $r$, such that $PR^{\mathcal{E}_i}_r$ is contained in the interior of $I^d$, for all $i$, and these pseudo-ring are disjoint. 
	
	For any such $r$, define $F$ as follows. In $PR^{\mathcal{E}_i}_r$, for $i=1,\ldots,N$, $F$ is equal to to the function given by Proposition \ref{translacao} for $\mathcal{E}_i$ and $r$, and equal to the identity outside $\cup_i PR^{\mathcal{E}_i}_r$.

	Note that $\|F-Id\|_\infty$ is less than or equal to the largest of the diameters of $PR^{\mathcal{E}_i}_r$, which is equal to $\max_i\{|P_i-Q_i|\}+2r$, as $P_1$ and $Q_i$ are the vertices of the ellipse.

	On the other hand, for $j=1,\ldots d$,
	\begin{equation*}
		\|(F-Id)_j\|_p = \left(\sum_{i=1}^N\int_{PR^{\mathcal{E}_i}_r}\left|(F-Id)_j\right|^pdx\right)^\frac{1}{p}\leq \|F-Id\|_\infty\sum_{i=1}^N\lambda\left( PR^{\mathcal{E}_i}_r\right)^\frac{1}{p},\\
	\end{equation*}
	and, as for $j,k=1,\ldots d$, $\left|\frac{\partial F}{\partial x_j}(x)\right|\leq \frac{C}{r}$ for some $C>0$, we have
	\begin{align*}
		\left\|\dfx{(F-Id)_j}{x_k}\right\|_p^p = & \sum_{i=1}^N\int_{PR^{\mathcal{E}_i}_r}\left|\dfx{(F-Id)_j}{x_k}\right|^pdx\leq \sum_{i=1}^N\int_{PR^{\mathcal{E}_i}_r}\left(\tfrac{C}{r}+1\right)^pdx\\
		= &\left(\tfrac{C}{r}+1\right)^p\sum_{i=1}^N\lambda\left( PR^{\mathcal{E}_i}_r\right).
	\end{align*}
	
	Therefore, recalling \eqref{volumePT} and that $p<d-1$, we get $\lim_{r\rightarrow 0}\left\|F-Id\right\|_{1,p}=0$.
	
	\vskip 2mm
	
	{\bf Case 2.} Suppose now that $d=2$. If the line segments connecting $P_i$ to $Q_i$ are disjoint, then we are in the first case. Otherwise define, for each $i=1,\ldots,N$, $f_i(t)=P_i+t(Q_i-P_i)$, for $t\in [0,1]$. Consider the set $\mathcal{A}$ of all intersection of those line segments, suppose it is a finite set with $m$ elements and that, for all $t\in [0,1]$,  $f_i(t)\neq f_j(t)$ if $i\neq j$. Let $t_1<\cdots <t_m$ be all the values of $t$ for which $f_i(t_j)\in \mathcal{A}$ for some $i$. Consider $0=s_0\leq t_1<s_1\cdots< s_{m-1}<t_m\leq s_m=1$. 
	
	For fixed $j=0,\ldots,m-1$, the line segments $f_i([s_j,s_{j+1}])$ are disjoint and so we can apply the first case for the families $ \{f_i(s_j))\}_{i=1}^N$ and $\{f_i(s_{j+1})\}_{i=1}^N$. Let $H_j$ be the volume preserving $C^{\infty}$ diffeomorphism obtained from {\bf Case 1.} and $PR^{\mathcal{E}_{i,j}}_{r_j}$, $i=1,\ldots,N$, the  pseudo-ring used in its construction.

	Let us see that $H=H_m\circ\cdots\circ H_1$ satisfies the desired conditions. The ones regarding $\|H-Id\|_\infty$ and $\|H-Id\|_p$ are consequences of the inequalities
	\begin{align*}
		\|H-Id\|_\infty\leq &\sum_{j=1}^m \|H_j-Id\|_\infty\\
		\leq &\sum_{j=1}^m\big(\max_i|P_i-Q_i|(s_{j+1}-s_j)+2r_j\big)\\
		&=
		\max_i|P_i-Q_i|+2\sum_{j=1}^mr_j,
	\end{align*}
	and
	\begin{equation*}
		\|(H-Id)_j\|_p \leq \|H-Id\|_\infty\sum_{i=1}^N\lambda\left( PR^{\mathcal{E}_{i,1}}_{r_j}\right)^\frac{1}{p}.
	\end{equation*}
	
	We will now obtain the control of the norm of the derivative of $H-Id$. To simplify the notation, consider $\mathcal{F}_j=\bigcup_{i=1}^N PR^{\mathcal{E}_{i,j}}_{r_j}$ and $\mathcal{G}_j=\mathcal{F}_j\setminus\bigcup_{l<j}\mathcal{F}_l$
	. There exists $C>0$ such that
	\begin{align*}
		\left\|\dfx{(H-Id)_j}{x_k}\right\|_p^p & =  \sum_{j=1}^m\int_{\mathcal{G}_j}\left|\dfx{(H_m\circ \cdots\circ H_j-Id)_j}{x_k}\right|^pdx\\
		\leq & \sum_{j=1}^m\int_{\mathcal{G}_j}\left(\frac{C}{r_j\cdots r_m}+1\right)^pdx=\sum_{j=1}^m\left(\frac{C}{r_j\cdots r_m}+1\right)^p\lambda({\mathcal{G}_j}).
	\end{align*}
	Therefore, for $r_i$ small enough, there exists $C'>0$, depending only on the data, such that
	
	\begin{equation*}
		\left\|\dfx{(F-Id)_j}{x_k}\right\|_p^p  \leq   C'\sum_{j=1}^m\frac{r_j^{d-1}}{\left(r_j\cdots r_m\right)^p}.
	\end{equation*}
	Choose $c_m,c_{m-1},\ldots,c_1>0$ sequentially such that 
	\begin{equation}\label{susto}
		\left\{
		\begin{array}{l}
			1+c_m<\tfrac{d-1}{p}\\[2mm] 
			1+c_{m-1}+c_{m-1}c_m<\tfrac{d-1}{p}\\[2mm]
			1+ c_{m-2}+c_{m-2}c_{m-1}+c_{m-2}c_{m-1}c_m<\tfrac{d-1}{p}\\ [2mm]
			\vdots \\
			1+ \sum_{t=j}^m\prod_{s=j+1}^tc_s<\tfrac{d-1}{p}\\
			\vdots
		\end{array}\right.
	\end{equation}
	For $r_1$ small enough, define $r_2=r_1^{c_2}$, $r_3=r_2^{c_3}$,\ldots, $r_m=r_{m-1}^{c_m}$, and notice that, for $j=1,\ldots,m$,
	\begin{equation*}
		\frac{r_j^{d-1}}{\left(r_j\cdots r_m\right)^p}=\frac{r_j^{d-1}}{\left(r_jr_j^{c_{j+1}}\cdots r_j^{c_{j+1}\cdots c_m}\right)^p}=r_j^{d-1-p\left( 1+ \sum_{t=j}^m\prod_{s=j+1}^tc_s\right)},
	\end{equation*}
	which converges to $0$, as $r_1$ tends to $0$ considering \eqref{susto}.
	
	If the assumptions about the set $\mathcal{A}$ referred in the beginning of \textbf{Case 2} are not satisfied, then we can consider small perturbations replacing some line segments with 
	two connected line segments such that the new set of points satisfies the conditions. In this case some of the functions $H_i$ will be substituted by a composition of two function of the same type. This argument follows \cite[Theorem 2.4]{AP} with the addition of the control of the derivative, which is unchanged by the perturbations.
\end{proof}

\section{Proof of Volume preserving Sobolev weak Lusin theorem (Theorem~\ref{crucial})}\label{PLusin}

In the present section we obtain a volume preserving weak Lusin theorem for the Sobolev class $\lipp{X}$ with $0<p<1$ (Theorem~\ref{crucial}).  Since any smooth closed connected  $d$-dimensional manifold can be obtained 
from $I^d$ by making boundary identifications, it follows, using a Moser's result~\cite{M}, that the proof of  Theorem~\ref{crucial} reduces to the proof on the unit cube $I^d$.

Firstly, we obtain a control of the partial derivatives of a certain homeomorphism of a cube that will be used afterwards.

\begin{lemma}\label{cubinho}
	Let $\sigma$ be a cube with edges parallel to the coordinates axis and, for $t\in\,]0,1[$ let $\sigma^t$ be the cube concentric with $\sigma$, with faces parallel to it's and such that $\lambda(\sigma^t)=t\lambda(\sigma)$.
	
	For $r,s\in\,]0,1[$ there exists an homeomorphism $T^{\sigma,r,s}:\sigma\rightarrow \sigma$, $C^\infty$ except in the boundary of $\sigma^r$, that is a dilatation from $\sigma^r$ onto $\sigma^s$
	and its inverse is $T^{\sigma,s,r}$.
	
	Furthermore, the determinant of the Jacobian of $T^{\sigma,r,s}$ is equal to $\frac{s}{r}$, in the interior of $\sigma_r$, and is equal to $\frac{1-s}{1-r}$ in  $\sigma\setminus\sigma^r$,  and
	\begin{equation}\label{cubolimitado}
		\forall i,j=1,\ldots, d\quad \forall x \notin \sigma^r \quad \left|\frac{\partial T^{\sigma,r,s}_i(x)}{\partial x_j}\right|\leq\left\{\begin{array}{ll}
			\left(\frac{s}{r}\right)^\frac{1}{d} & \text{if $r<s$}\\[2mm]
			\frac{(1-s)r}{(1-r)s}\left(\frac{s}{r}\right)^\frac{1}{d} & \text{if $s<r$.}
		\end{array}	\right.
	\end{equation}
	In particular, $T^{\sigma,r,s}$ and its inverse are Lipschitz.
\end{lemma}

\begin{proof} 
	Of course we can suppose that $\sigma=[-1,1]^d$. For  $r,s\in\,]0,1[$ and $A=\frac{1- s}{1- r}$ consider the function 
	\begin{equation*}
		\begin{array}{rclc}
			T^{\sigma, r,s}: & [-1,1]^d & \longrightarrow & [-1,1]^d\\
			& x        & \mapsto         & \left\{\begin{array}{cl}
				\frac{\sn}{\rn}\,x & \text{if $\|x\|_\infty\leq \rn$}\\[2mm]
				\frac{\left(A\|x\|_\infty^d+1-A\right)^{\frac{1}{d}}}{\|x\|_\infty}\,x & \text{otherwise.}
			\end{array}\right.
		\end{array} 
	\end{equation*}
	Notice that, for $\|x\|_\infty\geq \rn$, $	\|T^{\sigma,r,s}(x)\|_\infty=\left(A\|x\|_\infty^d+1-A\right)^{\frac{1}{d}}$ increases in  $\|x\|_\infty$ from $\sn$ to $1$. Indeed,
	\begin{equation*}
		\|T^{\sigma,r,s}(x)\|_\infty= \left(A\|x\|_\infty^d+1-A\right)^{\frac{1}{d}}\left\{\begin{array}{l}
			\geq \left(Ar+1-A\right)^{\frac{1}{d}}=\sn\\[2mm]
			\leq \left(A+1-A\right)^\frac{1}{d}=1.
		\end{array}\right.
	\end{equation*}
	
	If $\|x\|_\infty\geq \rn$ then 
	\begin{equation*}
		T^{\sigma,s,r}\left(T^{\sigma,r,s}(x)\right) =	\frac{\Big(\frac{1}{A}\left(A\|x\|_\infty^d+1-A\right)+1-\frac{1}{A}\Big)^{\frac{1}{d}}}{\left(A\|x\|_\infty^d+1-A\right)^{\frac{1}{d}}}\,\frac{\left(A\|x\|_\infty^d+1-A\right)^{\frac{1}{d}}}{\|x\|_\infty}\,x=x.
	\end{equation*}

	From now on we denote $T^{\sigma,r,s}$ by $T$.	For the second part of the proof we only need to consider points whose coordinates are all positive and their first coordinates are strictly bigger than the others. Then, in a neighbourhood of such points,
	\begin{equation*}
		T(x)=\left(\lambda(x_1), g(x_1)x_2,\ldots, g(x_1)x_d\right),
	\end{equation*}
	where $\lambda(x_1)=\left(Ax_1^d+1-A\right)^{\frac{1}{d}}$ and $g(x_1)=\frac{\lambda(x_1)}{x_1}$.

	The entries of the Jacobian of $T$ are then
	\begin{equation*}
		\lambda'(x_1),\ g(x_ 1),\   g'(x_1)x_i,\ \ \text{with $i>1$}.
	\end{equation*}

	Noticing that $ r< s$ if and only if $A<1$ and
	\begin{align*}
		\lambda'(x_1) & =  Ax_1^{d-1}\lambda(x_1)^{1-d}\\
		\lambda''(x_1) & =  (1-A)(d-1)Ax_1^{d-2}\lambda(x_1)^{1-2d}\\
		g'(x_1) & =  \frac{(A-1)\lambda(x_1)^{1-d}}{x_1^2} \\
		(g'(x_1)x_1)' & =  \frac{(1-A)\lambda(x_1)^{1-d}\Big(A(d+1)x_1^{d}\lambda(x_1)^{1-d}+1\Big)}{x_1^2},
	\end{align*} 
	we conclude that $|g'(x_1)|x_1$ is decreasing, $\lambda'$ is increasing if $A<1$ and decreasing if $A>1$, and $g$ is decreasing if $A<1$ and increasing if $A>1$.
	Then, recalling that $|x_1|\leq x_i$, for $i>1$,
	\begin{equation*}
		\left\{
		\begin{array}{ll}
			\lambda'(x_1) & \leq 1\\
			g(x_1)        & \leq  \left(\frac{ s}{ r}\right)^{\frac{1}{d}}\\
			|g'(x_1)x_i|      & \leq  \frac{ s- r}{(1- r)s}\left(\frac{ s}{ r}\right)^{\frac{1}{d}}
		\end{array}
		\right.
		\text{ if $ r< s$,}\qquad\left\{
		\begin{array}{ll}
			\lambda'(x_1) & \leq \frac{(1-s)r}{(1-r)s}\left(\frac{ s}{ r}\right)^{\frac{1}{d}}\\[3mm]
			g(x_1)        & \leq  1\\
			|g'(x_1)x_i|      & \leq \frac{ r-s}{(1- r)s}\left(\frac{ s}{ r}\right)^{\frac{1}{d}}
		\end{array}
		\right.
		\text{ if $ r> s$,}
	\end{equation*}
	from where we obtain the desired inequalities.
	
	Finally, $\det (JT(x))=\lambda'(x_1)g(x_1)^{d-1}=\lambda'(x_1)\frac{\lambda^{d-1}(x_1)}{x_1^{d-1}}=A$, since the Jacobian of $T(x)$ is a lower triangular matrix, 
\end{proof}

We recall that a {\em dyadic permutation of $I^d$ of order $m$} is a bijection $\mathscr{P} \colon I^d \to I^d$ which permutes by simple translation the  dyadic open cubes that are products of intervals of the form $]k/2^m,(k+1)/2^m[$.

\bigskip

\begin{theorem}\label{key:lusin}
	Let 	$d \geq 2$, $0<p<1$ and  $\mathscr{P}$ be a dyadic permutation of the cube $I^d$.	
	Given any $\delta,\gamma>0$, there is 
	$f\in\lip{I^d}$,  equal to the identity on a neighbourhood of the boundary of $I^d$, satisfying 
	\begin{equation*}
		\|f-Id\|_\infty< \|\mathscr{P}-Id\|_\infty+\delta,\quad
		\|D(f-Id)\|_p<\delta, \quad and\quad \lambda\{x \colon \mathscr{P}(x) \neq f(x) \} < \gamma.
	\end{equation*}
	
	Moreover, $f$ is $C^\infty$ except in a set with zero measure.
\end{theorem}

\begin{proof} Without loss of generality, we can suppose that $\mathscr{P}$ is a permutation of dyadic cubes $\sigma_i$, $i=1,\ldots,N$, with diameter less than $\frac{\delta}{3}$ (just increase the order of the dyadic permutation). For $0<\beta<1$, recall that we denote by $\sigma_i^{\beta}$ the cubes concentric to $\sigma_i$ with parallel faces and such that $\lambda(\sigma_i^{\beta}) = \beta \, \lambda(\sigma_i)$.
	We denote by $P_i$ the centre of the cube $\sigma_i$ and let  $Q_i=\mathscr{P}(P_i)$, $i =1, \ldots,N$. Since $\|\mathscr{P}-Id\|_{\infty} < \|\mathscr{P}-Id\|_{\infty}+\frac{\delta}{3}$, applying Theorem~\ref{key:pert2} to the sets $\{P_i\}_{i=1}^{N}$ and $\{Q_i\}_{i=1}^{N}$, we obtain 
	a volume preserving $C^\infty$ diffeomorphism of $I^d$, $F$, 
	equal to the identity on a neighbourhood of	the boundary of $I^d$ and such that
	\begin{equation*}
		\|F-Id\|_\infty <\|\mathscr{P}-Id\|_{\infty}+\tfrac{\delta}{3},\quad \|F-I\|_{1,p}<\tfrac{\delta}{3}.
	\end{equation*}
	Furthermore, since
	$F$ sends a neighbourhood of $P_i$ by translation onto a neighbourhood of $Q_i$,  
	there exists $0<\alpha<1$ such that $F=\mathscr{P}$ on $\sigma_i^{\alpha}$, for every $i$. Consequently, 
	$\lambda\{x : \, \mathscr{P}(x) \neq F(x) \} \leq 1 - \alpha$.
	If $1-\alpha<\gamma$, the theorem is proved taking $f=F$.
	Otherwise, it is clearly enough to obtain a map $f$ which, in addition to satisfying the same conditions, coincides with $\mathscr{P}$ on the cubes $\sigma_i^{\beta}$, for some $\beta$ such that $\beta>1-\gamma$. 
	Fix $\beta >1-\gamma$ and define a map
	$T \colon I^d \to I^d$ (not volume preserving) such that its restriction to each cube $\sigma_i$ is $T^{\sigma_i,\alpha,\beta}$, given by Lemma \ref{cubinho}.

	Finally, we define the map $f:=T\circ F\circ  T^{-1}$. This map satisfies the following:
	\begin{enumerate}
		
		\item[(i)] $\lambda\{x \colon \mathscr{P}(x) \neq f(x) \} < \gamma$, by construction.
		
		\item[(ii)]  $x \in \cup_i \, \text{int}(\sigma_i^\beta)$ if and only if $F(T^{-1}(x)) \in  \cup_i \, \text{int} (\sigma_i^\alpha)$.
		
		\item[(iii)] $f$ is volume preserving. In fact, by~(ii), 
		\begin{equation*}
			\det(J f(x))=\det(JT(F(T^{-1}(x))) \det(JF(T^{-1}(x))\det(JT^{-1}(x))=1,
		\end{equation*}
		except in the boundary of $\sigma_i$, $\sigma_i^\alpha$ and $\sigma_i^\beta$.
	\end{enumerate}

	It is clear that $\|f-Id\|_{\infty} \leq \|F-Id\|_{\infty}+2 \|T-Id\|_{\infty}<\|\mathscr{P}-Id\|_\infty+\delta$. 
	
	As $f=Id$ in $\cup \sigma_i^\beta$ then $\|f-Id\|_p\leq \|f-Id\|_{\infty}(1-\beta)^\frac{1}{p}<\delta$, for $\beta$ large enough.

	Using \eqref{pq} and the fact that $f=Id$ in $\sigma_i^\beta$ we get,
	\begin{equation*}
		\left\|\dfx{(f-Id)_i}{x_j}\right\|_p^p =
		\left\|\dfx{(f-Id)_i}{x_j}\right\|_{L^p(I^d\setminus \cup \sigma_i^\beta))}^p \leq  \left\|\dfx{f_i}{x_j}\right\|_{L^p(I^d\setminus \cup \sigma_i^\beta))}^p	
		+
		\left\|\dfx{Id_{i}}{x_j}\right\|_{L^p(I^d\setminus \cup \sigma_i^\beta))}^p.	
	\end{equation*}
	Notice that
	\begin{equation*}
		\left\|\dfx{Id_{i}}{x_j}\right\|_{L^p(I^d\setminus \cup \sigma_i^\beta)}^p	\leq 	\lambda(I^d\setminus \cup \sigma_i^\beta)= 1-\beta.
	\end{equation*}
	
	On the other hand, using the chain rule and \eqref{cubolimitado} for $T$ and $T^{-1}$, we obtain, for $i,j=1,\ldots,d$,
	\begin{align*}
		\left|\dfx{f_i}{x_j}(x)\right|  \leq &  \frac{(1-\alpha)\beta}{(1-\beta)\alpha}\left(\frac{\alpha}{\beta}\right)^\frac{1}{d}\left(\frac{\beta}{\alpha}\right)^\frac{1}{d}\sum_{k,l=1}^d\left|\dfx{F_k}{x_l}(T^{-1}(x))\right|\\
		\leq & \frac{1}{(1-\beta)\alpha} \sum_{k,l=1}^d\left|\dfx{F_k}{x_l}(T^{-1}(x))\right|.
	\end{align*}
	Then, using \eqref{pq}, (ii), the change of variables defined by $T^{-1}$ and the fact that  $F=\mathscr{P}$ in $\sigma^\alpha_i$, we obtain
	\begin{align*}
		\left\|\dfx{f_i}{x_j}\right\|_{L^p(I^d\setminus \cup \sigma_i^\beta))}^p	& \leq  \frac{1}{(1-\beta)^p\alpha^p}	\sum_{k,l=1}^d\int_{I^d\setminus \cup \sigma_i^\beta} \left|  \dfx{F_k}{x_j}  (T^{-1}(x)) \right|^p\,dx\\ 
		&= \frac{1}{(1-\beta)^p\alpha^p}\frac{1-\beta}{1-\alpha}\sum_{k,l=1}^d  \int_{I^d\setminus \cup \sigma_i^\alpha} \left|  \dfx{F_k}{x_j}(x)\right|^p  \,dx \\
		& \leq \frac{(1-\beta)^{1-p}}{\alpha^p(1-\alpha)}d^2 \left\|DF\right\|_{L^p(I_d)}^p \,.
	\end{align*}

	Hence, taking $\beta$ large enough, the conclusion follows.
\end{proof}

\bigskip

We can now complete the proof of Theorem~\ref{crucial}.

\begin{proof} (of Theorem~\ref{crucial}) 
	Let  $g \in \mathcal{G}_\lambda(I^d)$, $h\in\lipp{I^d}$ and consider $\delta>0$. Let $\mathcal{W}$ be any weak topology neighbourhood of $g$.
	We will obtain $f \in\lip{I^d}$ such that  $f \in \mathcal{W}$ and 
	\begin{equation*}
		\|f-\pi_1(h)\|_\infty < \|g - \pi_1(h)\|_{\infty}+\delta,\qquad 	 \|Df-\pi_2(h)\|_p < \delta.
	\end{equation*}
	Recall that, if $h\in\lip{I^d}$, then $\pi_1(h)=h$ and $\pi_2(h)=Dh$.

	The proof will be divided into three steps.
	
	\medskip
	
	{\bf Step 1.} Assume first that $h=Id$. We only need to consider  $\mathcal{W}$ of the form
	\begin{equation*} 
		\left\{\tilde{g} \in \mathcal{G}_\lambda(I^d): \, \lambda \{x: \, |g(x)-\tilde{g}(x)| \geq \gamma\}<\gamma\right\},\quad\text{for $\gamma>0$}.
	\end{equation*}

	Since the dyadic permutations are dense in $\mathcal{G}_\lambda(I^d)$, the automorphism $g\in \mathcal{G}_\lambda(I^d)$ can be weakly arbitrarily approximated by a dyadic permutation $R$ (see~\cite[Lemma 6.4]{AP}).
	Moreover, $R$ can be weakly approximated by another dyadic permutation $\mathscr{P}$ satisfying:
	\begin{equation}\label{pgp}
		\lambda \{x: \, |g(x)-\mathscr{P}(x)| \geq \gamma\} < \gamma\,\,\,\text{ and }\,\,\,\|\mathscr{P}-Id\|_{\infty} < \|g-Id\|_\infty+\tfrac{\delta}{2}.
	\end{equation}
	The technique for this approximation is described in the proof 
	of~\cite[Theorem 6.2, pp.\,46]{AP}.

	Set $\gamma_0:=\gamma - \lambda \{x : |g(x)-\mathscr{P}(x)| \geq \gamma\}$.
	Applying Theorem~\ref{key:lusin} (with $\frac{\delta}{2}$) to the  permutation $\mathscr{P}$ and using \eqref{pgp} we obtain a map 
	$f \in \lip{I^d}$, with 
	$\|f-Id\|_\infty < \|\mathscr{P}-Id\|_\infty+\frac{\delta}{2}<\|g-Id\|_\infty+\delta$, $\|D(f-Id)\|_p<\frac{\delta}{2}$ and equal to the identity on a neighbourhood of the boundary, satisfying 
	$\lambda\{x \colon \mathscr{P}(x) \neq f(x) \} < \gamma_0$.
	From this last point we also conclude that $f \in \mathcal{W}$.

	\medskip
	
	{\bf Step 2.} Assume now that $h\in \lip{I^d}$ and consider $\mathcal{W}$ a  weak topology neighbourhood  of $g$ in $\mathcal{G}_\lambda(I^d)$.
	
	As $\mathcal{G}_\lambda(I^d)$, with the weak topology, is a topological group then
	$\mathcal{W}_h=\{\varphi\circ h^{-1}:\varphi\in\mathcal{W}\}$ is a weak neighbourhood of $g\circ h^{-1}$ and
	\begin{equation*}
		\|g\circ h^{-1}-Id\|_\infty=\sup_x\|g\left(h^{-1}(x)\right)-h\left(h^{-1}(x)\right)\|_\infty=\|g-h\|_\infty.
	\end{equation*}
	Using {\bf Step 1.}, consider $f^*\in  \lip{I^d}$ (which is in fact $C^\infty$ almost everywhere) such that  $f^*\in \mathcal{W}_h$ and
	\begin{equation*}
		\|f^*-Id\|_{\infty} <  \|g\circ h^{-1} - Id\|_{\infty}+\delta=\|g- Id\|_{\infty}+\delta^*,\quad \|D(f^*-Id)\|_p<\delta^*,
	\end{equation*}
	for $\delta^*>0$ to be defined later.
	
	Consider $f=f^*\circ h$. Notice that $f\in\mathcal{W}$, $f\in\lip{I^d}$, $\|f-h\|_\infty=\|f^*-Id\|_\infty$ and, for $i,j=1,\ldots,d$, and $M=\max_{i,j}\left\|\frac{\partial h_i}{\partial x_j}\right\|_\infty$,
	\begin{equation*}
		\left|\frac{\partial}{\partial x_j} (f-h)_i(x)\right|\leq M\sum_{k=1}^d	\left|\frac{\partial}{\partial x_k} (f^*-Id)_i(h(x))\right|.
	\end{equation*}
	
	Then, as $h$ preserves the volume measure $\lambda$, 	
	\begin{align*}
		\left\|\frac{\partial}{\partial x_j} (f-h)_i\right\|_p^p & \leq  M^p\int\left(\sum_{k=1}^d	\left|\frac{\partial}{\partial x_k} (f^*-Id)_i(h(x))\right|\right)^p\,dx \\
		& =  M^p\int\left(\sum_{k=1}^d	\left|\frac{\partial}{\partial y_k} (f^*-Id)_i(y)\right|\right)^p\,dy\\
		& =  M^p\left\|\sum_{k=1}^d	\left|\frac{\partial}{\partial y_k} (f^*-Id)_i\right|\right\|_p^p.
	\end{align*}
	Therefore, using \eqref{pq}, 
	\begin{equation*}
		\left\|\frac{\partial}{\partial x_j} (f-h)_i\right\|_p^p  \leq  M^p\sum_{k=1}^d	\left\|\frac{\partial}{\partial y_k} (f^*-Id)_i\right\|_p^p \leq  M^pd \|D(f^*-Id)\|_p^p<
		M^pd \left(\delta^*\right)^p.
	\end{equation*}
	The conclusion follows choosing $\delta^*$ conveniently.
	
	\medskip
	
	{\bf Step 3.} Suppose now that $h$ is the limit of a sequence $(h_n)_n$ in $\lip{I^d}$. Let $\delta^*>0$ to be defined later, and  $n_0$ such that $\|\pi_1(h)-h_{n_0}\|_\infty,\|\pi_2(h)-Dh_{n_0}\|_p< \delta^*$. 
	
	Applying {\bf Step 2}, consider $f\in \lip{I^d}$ such that $f\in\mathcal{W}$ and
	\begin{equation*}
		\|f-h_{n_0}\|_\infty <\|g-h_{n_0}\|_\infty+\delta^*,\qquad \|Df-Dh_{n_0}\|_p<\delta^*.
	\end{equation*}
	
	Then, as $\|\pi_1(h)-h_{n_0}\|_\infty <\delta^*$,
	\begin{align*}
		\|f-\pi_1(h)\|_\infty & \leq  \|f-h_{n_0}\|_\infty+\|h_{n_0}-\pi_1(h)\|_\infty\\
		& <  \|g-h_{n_0}\|_\infty+2\delta^*\\
		& \leq   \|g-\pi_1(h)\|_\infty+ \|\pi_1(h)-h_{n_0}\|_\infty+2\delta^*\\
		& < \|g-\pi_1(h)\|_\infty+3\delta^*
	\end{align*}
	and, as $\|\pi_2(h)-Dh_{n_0}\|_p <  \delta^*$, using \eqref{pq},
	\begin{equation*}
		\|Df-\pi_2(h)\|_p^p \leq \|Df-Dh_{n_0}\|_p^p+\|Dh_{n_0}-\pi_2(h)\|_p^p< 2\delta^*
	\end{equation*}
	and the conclusion follows if we choose $\delta^*=\min\left\{\frac{\delta}{3}, \frac{\delta^p}{2}\right\}$.
\end{proof}

\begin{remark}
	The result of Theorem \ref{key:lusin} and, consequently, of Theorem \ref{crucial} is false if $p\geq 1$. To see that, consider any permutation $\mathscr{P}$ different from the identity. Suppose that, for $\delta,\gamma>0$,
	there exists $f\in\lip{I^d}$,  equal to the identity on a neighbourhood of the boundary of $I^d$, satisfying 
	\begin{equation*}
		\|f-Id\|_\infty< \|\mathscr{P}-Id\|_\infty+\delta,\quad
		\|D(f-Id)\|_p<\delta, \quad and\quad \lambda\{x \colon \mathscr{P}(x) \neq f(x) \} < \gamma.
	\end{equation*}
	
	As $f-Id$ equals to zero in the boundary of $I^d$ then,  using the Poincaré inequality, there exists $C$, depending on $d$ and $p$, such that 	$\|f-Id\|_p	\leq C\|D(f-Id)\|_p$. On the other hand 
	\begin{equation*}
		\|\mathscr{P}-Id\|_p\leq \|\mathscr{P}-f\|_p+\|f-Id\|_p\leq 2\gamma^\frac{1}{p}+a\delta,
	\end{equation*}
	which is a contradiction, as we can choose $\delta,\gamma$ as small as we want.
\end{remark}

\section{Proof of Sobolev Oxtoby-Ulam theorem (Theorem \ref{OU})}\label{PMR}

In order to prove the Sobolev Oxtoby-Ulam theorem, we need an auxiliary lemma that is the adaptation of \cite[Lemma 7.2]{AP} to our setting. A crucial ingredient is our version of Lusin Theorem (Theorem~\ref{crucial}). As noticed before the proof will be done in the cube $I^d$.

\begin{lemma}\label{lemma} 
	Let $d\geq 2$, $0<p<1$ and $h\in \lipp{I^d}$. Then, given any $\eps>0$, there exists $f\in \lip{I^d}$ satisfying $\|f-h\|_{1;\infty,p}<\eps$ and such that, for some arbitrarily fine dyadic decomposition $\{D_i\}_{i=1}^N$, there is a compact set $B\subset D_1$ with $\lambda(B)=\frac{1}{2}\lambda(D_1)$ and $f^i(B)$ in the interior of the dyadic cube $D_{i+1}$ for $i=0,\dots, N-1$.
\end{lemma}

\begin{proof}
	By \cite[Theorem 3.3]{AP} there is a cyclic dyadic permutation
	$\mathscr{P}$ of some dyadic decomposition $D_1,\ldots,D_N$ with $\mathscr{P}(D_i)=D_{i+1}$, 
	for $i=1,\ldots,N-1$ and $\mathscr{P}(D_N)=D_1$, with $\|\mathscr{P}-\pi_1(h)\|_\infty<\eps$. For $D_i=\prod_{j=1}^d\big]a^i_j-\delta,a^i_j+\delta\big[$, with $\delta=\frac{1}{2\sqrt[d]{N}}$, define, for  $\alpha\in]0,1[$,  $D_{i,\alpha}=\prod_{j=1}^d\big]a^i_j-\alpha\delta,a^i_j+\alpha\delta\big[$. 
	
	By Theorem \ref{crucial}, there exists some $f\in\lip{I^d}$ with $\|f-h\|_{\infty;1,p}<\eps$, arbitrarily close to $\mathscr{P}$ in the weak topology. Therefore, there is a sequence $\left(f_n\right)_n$  in $\lip{I^d}$ with $\|f_n-h\|_{\infty;1,p}<\eps$, which converges to $\mathscr{P}$ in the weak topology. 
	
	In particular, for $a$ to be specified later, and $\mu\in]0,1[$ satisfying 
	$\lambda(D_{1,\mu})>\frac{3}{4}\lambda(D_1)$, there exists $n_0$  such that $\lambda\left(f_{n_0}(D_{i,\mu})\triangle \mathscr{P}(D_{i,\mu})\right)< a$, for all $i=1,\ldots,N-1$. Note that, as $\mathscr{P}(D_{i,\mu})=D_{i+1,\mu}$,  $\lambda\left(f_{n_0}(D_{i,\mu})\setminus D_{i+1,\mu}\right)< a$. 
	
	Let $E_1=D_{1,\mu}$ and, for each $i>1$, $E_i=f(E_{i-1})\cap D_{i,\mu}$. By construction, $E_i$ is contained in the interior of $D_i$,  $\lambda(E_i)> \lambda(D_{1,\mu})-(i-1)a$ and then, choosing $a<\frac{1}{4(N-1)}$, we have $\lambda(E_N)>\frac{1}{2}\lambda(D_1)$. 
	
	Consider $B^*=\left(f^{N-1}\right)^{-1}(E_N)$. Note that $B^*$ is a compact subset of $D_1$ and satisfies $\lambda(B^*)>\frac{1}{2}\lambda(D_1)$ and $f^i(B^*)$ is included in $D_{i+1,\mu}$. 
	
	Define $g:[0,1[\rightarrow \R$ by $g(\alpha)=\lambda(B^*\cap D_{1,\alpha})$, which is a continuous function, since $\big|g(\beta)-g(\alpha)\big|\leq \left|\lambda(D_{1,\beta})-\lambda(D_{1,\alpha})\right|=\left|(2\beta)^N-(2\alpha)^N\right|$, for $\alpha,\beta\in ]0,1[$. In particular, there exists $\alpha^*$ such that $B=B^*\cap D_{1,\alpha^*}$ is a compact subset of $D_1$, such that its measure is equal to $\frac{1}{2}\lambda(D_1)$ and $f^i(B)$ is included in the interior of $D_{i+1}$, for all $i=0,\ldots,N-1$.
\end{proof}

\bigskip

\begin{proof} (of Theorem~\ref{OU})
	We exploit ideas developed in the proof of \cite[Theorem 7.1]{AP}, adapted with the many subtleties of the current Sobolev setting.
	
	Let $\mathcal{E}=\{f\in\ml{I^d}:\text{$f$ is ergodic}\}$ and $\mathcal{E}_\infty=\{f\in \mlp{I^d}: \pi_1(f) \text{ is ergodic}\}$.
	
	In \cite[Theorem 7.1]{AP}, it was proved that there exists $(\mathcal{F}_n)_{n\in\N}$ such that $\ml{I^d}\setminus\mathcal{E}=\cup_n\mathcal{F}_n$ and, for all $n$, $\mathcal{F}_n$ is nowhere dense, that is,
	\begin{equation*}
		\forall h\in\mathcal{F}_n\quad\ \forall \eps>0\quad \exists f\in B^\infty(h,\eps)\ \exists \delta>0:\quad B^\infty(f,\delta)\cap \mathcal{F}_n=\emptyset.
	\end{equation*}
	
	The function $f$ here is the one given by \cite[Lemma 7.2]{AP}, whose version in our context is Lemma \ref{lemma}.\vskip2mm
	
	Let $\mathcal{G}_n=\{h\in\mlp{I^d}:\pi_1(h)\in\mathcal{F}_n\}$. It is clear that $\mlp{I^d}\setminus\mathcal{E_\infty}=\cup_n\mathcal{G}_n$. To prove that $\mathcal{G}_n$ is nowhere dense, consider $h\in \mathcal{G}_n$ an $\eps>0$. Using  Lemma \ref{lemma} and the above, there exists $f\in \lip{I^d}$ and $\delta>0$ such that 
	\begin{equation*}
		\|h-f\|_{1;\infty,p}<\eps,\quad B^\infty(f,\delta)\cap \mathcal{F}_n=\emptyset.
	\end{equation*}
	
	If $g\in \mlp{I^d}$ and $\|g-f\|_{1;\infty,p}<\delta$ then, in particular, $\pi_1(g)\in B^\infty(f,\delta)$, from where we conclude that  $\pi_1(g)\not\in \mathcal{F}_n$, i.e., $g\not\in \mathcal{G}_n$.
\end{proof}

\bigskip

\section{Proof of genericity of topological transitivity (Theorem \ref{Trans})}\label{PTT}

Our proof of Theorem \ref{Trans} will be done in the cube $I^d$ and follows the same steps as \cite[Theorem 4.1]{AP} using as a novelty our Theorem~\ref{key:pert2} where we deal with a finer topology.

Since, in this section $p\geq 1$, an element $f\in\mlp{I^d}$ is identified with $\pi_1(f)$, which is a volume preserving homeomorphism admitting weak partial derivatives in $L^p(I^d)$.

\begin{proof}(of Theorem \ref{Trans})
	Let $\mathcal{Y}$ be the (enumerable) set of the open dyadic cubes of any size and consider, for $A,B\in\mathcal{Y}$,
	\begin{equation*}
		\mathcal{T}_{A,B}=\left\{f\in   \mlp{I^d}\colon f^n(A)\cap B\neq \emptyset\quad\text{for some $n\in\N$} \right\}.
	\end{equation*}
	If we denote by $\mathcal{T}$ the set of all topologically transitive elements of $\mlp{I^d}$ then, as any open set in $I^d$ contains a dyadic cube, 
	\begin{equation*}
		\bigcap_{A,B\in\mathcal{Y}}\mathcal{T}_{A,B}= \mathcal{T}.
	\end{equation*}
	
	As, for $A,B\in\mathcal{Y}$, $\mathcal{T}_{A,B}$ is open (as it is open in the uniform topology), we are left to show that it is also a dense subset of
	$\ml{I^d}$.	To do this, it is enough to consider any $f\in \lip{I^d}$ and $\eps>0$ and find $g\in\mathcal{T}_{A,B}$ such that 
	\begin{equation*}
		\|g-f\|_\infty , \quad \left\|D(g-f)\right\|_p<\eps.
	\end{equation*}
	
	By \cite[Theorem 3.3]{AP} there exists a cyclic dyadic permutation $\mathscr{P}$ of some dyadic decomposition $D_1,\ldots,D_N$, with $\mathscr{P}(D_k)=D_{k+1}$, 
	for $k=1,\ldots,N-1$ and $\mathscr{P}(D_N)=D_1$, with $\|f-\mathscr{P}\|_\infty<\eps/2$. We can also suppose that the size of these dyadic cubes is less than the sizes of $A$ and $B$. Let $c_k$ stand for the centre of $D_k$.
	
	We now apply Theorem~\ref{key:pert2} to the sets $ \{f(c_k)\}_{k=1}^N$ and $\{\mathscr{P}(c_k)\}_{k=1}^N$, obtaining  a $C^{\infty}$ volume preserving diffeomorphism $F$ of $I^d$  which is equal to the identity in a neighbourhood of the boundary of $I^d$, sends by translation a ball centred in $f(c_k)$  onto the ball centred in $\mathscr{P}(c_k)$ and, for $\delta=\delta(\eps)$ to be defined later,
	\begin{equation*}
		\|F-Id\|_\infty < \max_k|f(c_k)-\mathscr{P}(c_k)|+\tfrac{\eps}{2}<\eps,\qquad  \left\|F-Id\right\|_{1,p}<\delta.
	\end{equation*}
	Since $f\in \lip{I^d}$ and $F$ is a  $C^\infty$ volume preserving diffeomorphism we get $g:= F\circ f\in \lip{I^d}$. Moreover, 
	for all $k=1,\dots,N-1$,
	\begin{equation*}
		g(c_k)=(F\circ f) (c_k)=\mathscr{P}(c_k)=c_{k+1},\quad g(c_N)=(F\circ f)(c_N)=\mathscr{P}(c_N)=c_{1}.
	\end{equation*}
	Let $i,j$ be such that $c_i\in A$ and $c_j\in B$.  Then, as $g$ permutes in a cyclical way the points $c_1,\ldots,c_N$, there exists $n\in\N$ such that $g^n(c_i)=c_j$ and then $g^n(A)\cap B\not=\emptyset$. Hence $g\in \mathcal{T}_{A,B}$.\vskip1mm
	
	It is clear that $\|g-f\|_\infty=\|F\circ f-f\|_\infty=\|F-I_d\|_\infty<\eps$.
	
	Finally, to estimate the norm of the partial derivatives of $g-f$ we use the same reasoning as in \textbf{Step 2} of the proof of Theorem \ref{crucial}. As $g-f=(F-I)\circ f$, using the chain rule, we obtain  $\|D(g-f)\|_p\leq M\,\|D(F-I)\|_p<M\delta<\eps$, for a certain $M>0$ depending on $\|Df\|_\infty$, and $\delta=\frac{\eps}{M}$.
\end{proof}

\section{Construction of an example}\label{example}

Now we are in conditions to exhibit $(g,h)\in\mlp{I^d}$, for $0<p<1$, such that $h$ is not the weak derivative of $g$, proving that $\mlp{I^d}$ is not a functional space.\vskip1mm

For $d\geq 2$, we will give an example of a sequence $(f_n)_n$ of measure preserving Lipschitz homeomorphism of $I^d$ converging uniformly to the identity $Id$ and such that the sequence $(Df_n)_n$ converges in $L^p(I^d)^{d^2}$, with $0<p<1$, to $DJ$, where $J:I^d\rightarrow I^d$ is defined by $J(x_1,x_2,x_3,\ldots,x_d)=(1-x_1,1-x_2,x_3,\ldots,x_d)$.\vskip2mm

It is enough to prove that, for all $\delta>0$, there exists $f$, a measure preserving Lipschitz homeomorphisms  of $I^d$, such that 
\begin{equation*}
	\|f-Id\|<\delta, \quad \left\|\dfx{f_i}{x_j}-\dfx{J_i}{x_j}\right\|_p\leq \delta, \quad i,j=1,\ldots,d.
\end{equation*}

It is clear that, if  we have such an $f$ for $d=2$, then $\tilde{f}$ defined as
\begin{equation*}
	\forall (x_1,x_2,x_3,\ldots,x_d)\quad \tilde{f}(x_1,x_2,x_3,\ldots,x_d)=(f(x_1,x_2),x_3,\ldots,x_d),
\end{equation*}
defines the intended function.\vskip2mm

So, let us consider $d=2$. Given $\delta>0$, consider a dyadic partition of $I^2$ such that the $\sigma_i$ squares of the partition have diameter less then $\delta$.
Define $F:[0,1]^2\rightarrow [0,1]^2$ such that, in each square $\sigma_i$ of the partition, $F$ is given by the restriction to $\sigma_i$ of the function $G$ given by Lemma \ref{rotacao2}, for $d=2$, $A$ the centre of $\sigma_i$ and $s=\frac{1}{2}\lambda(\sigma_i)^\frac{1}{2}$.

Consider a square $\sigma_i^\alpha$ inscribed in the circle centred in $A$ and radius $\frac{s}{2}$.
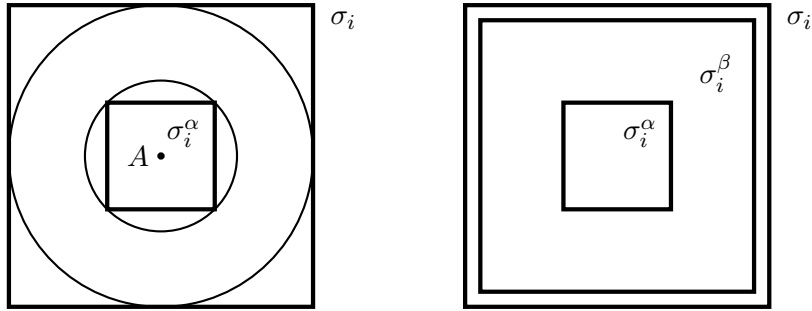
\begin{figure}[htb]
	\begin{center}
		\begin{tikzpicture}
			\draw [ultra thick](0,0) rectangle (4,4);
			\draw [ultra thick](1.2928,1.2928) rectangle (2.7071,2.7071);
			\draw[thick] (2,2) circle (2cm);
			\draw[thick] (2,2) circle (1cm);
			\node at (4.4,3.8) {$\sigma_i$};
			\node at (2.3,2.3) {$\sigma_i^\alpha$};
			\draw[ultra thick] (2,2) circle (0.2mm);
			\node at (1.7,2) {$A$};
				\draw [ultra thick](6,0) rectangle (10,4);
				\draw [ultra thick](7.2928,1.2928) rectangle (8.7071,2.7071);
				\draw [ultra thick](6.2,0.2) rectangle (9.8,3.8);
				\node at (8.3,2.3) {$\sigma_i^\alpha$};
				\node at (9.3,3.1) {$\sigma_i^\beta$};
				\node at (10.4,3.8) {$\sigma_i$};
			\end{tikzpicture}
		\end{center}
		\caption{An illustration of the example.}
		\label{key:exemplo}
	\end{figure}
	
	For $\alpha\leq \beta\leq 1$, let $T_\beta:[0,1]^2\rightarrow [0,1]^2$ be a homeomorphism such that, in each $\sigma_i$, $T^\beta$ is, as in Lemma \ref{cubinho}, a dilatation from $\sigma_i^\alpha$ to $\sigma_i^\beta$ (see Figure \ref{key:exemplo}).
	Finally consider $f=T_\beta\circ F\circ T^{-1}_\beta$.
	
	As $f(\sigma_i)=\sigma_i$, then $\|f-Id\|_\infty<\delta$. 
	
	In what concerns to the derivative, as $f$ is a rotation of an angle of $\pi$ in each $\sigma_i^\beta$, then 
	\begin{equation*}
		\left\|\dfx{f_i}{x_j}-\dfx{J_i}{x_j}\right\|_p^p=	\left\|\dfx{f_i}{x_j}-\dfx{J_i}{x_j}\right\|_{L ^p(I^d\setminus \cup_i\sigma_i^\beta)}^p\leq \left\|\dfx{f_i}{x_j}\right\|_{L ^p(I^d\setminus \cup_i\sigma_i^\beta)}^p+\left\|\dfx{J_i}{x_j}\right\|_{L ^p(I^d\setminus \cup_i\sigma_i^\beta)}^p.
	\end{equation*}
	
	Following the last part of the proof of Theorem \ref{key:lusin}, we get the estimates
	\begin{align*}
		\left\|\dfx{J_{i}}{x_j}\right\|_{L^p(I^d\setminus \cup \sigma_i^\beta)}^p &	\leq 1-\beta,\\[2mm]
		\left\|\dfx{f_i}{x_j}\right\|_{L^p(I^d\setminus \cup \sigma_i^\beta))}^p	& \leq  \frac{1}{(1-\beta)^p\alpha^p} \frac{(1-\beta)^{1-p}}{\alpha^p(1-\alpha)}d^2 \left\|DF\right\|_{L^p(I_d)}^p.
	\end{align*}
	
	Choosing $\beta<1$ large enough we obtain  $\left\|\dfx{f_i}{x_j}-\dfx{J_i}{x_j}\right\|_p<\delta$.
	\bigskip


	
	
\section*{Acknowledgements}
AA, DA and MJT were partially financed by Portuguese Funds through FCT  - `Funda\c{c}\~ao para a Ci\^encia e a Tecnologia' within the research Project UID/00013: Centro de Matem\'atica da Universidade do Minho (CMAT/UM). MB was partially supported by CMUP, which is financed by national funds through FCT-Funda\c c\~ao para a Ci\^encia e a Tecnologia, I.P., under the project with reference UIDB/00144/2020 and also partially funded by the project ``Means and Extremes in Dynamical Systems'' PTDC/MAT-PUR/4048/2021.

\end{document}